\documentclass[10pt]{amsart}

\usepackage[english]{babel}
\usepackage{amsmath, amsfonts, amssymb, amsthm, epsfig, graphicx, url, float, subfigure, xcolor}
\usepackage[colorlinks=true,linkcolor=blue,citecolor=blue]{hyperref}

\newtheorem{teor}{Theorem}[section]
\newtheorem{lema}{Lemma}[section]
\newtheorem{corollary}{Corollary}[section]
\newtheorem{prop}{Proposition}[section]

\newtheorem{remark}{Remark}[section]
\newcounter{theor}

\theoremstyle{definition}

\def\K{\mathcal{K}}

\def\R{\mathbb{R}}
\def\N{\mathbb{N}}

\def\vol{\mathrm{vol}}

\def\C{\mathbb{C}}
\def\S{\mathrm{S}}

\def\W{\mathrm{W}}

\def\inr{\mathrm{r}}

\numberwithin{equation}{section}

\begin{document}

\title[A reverse isoperimetric ineq. for relative outer parallel bodies]{On a reverse isoperimetric inequality for relative outer parallel bodies}

\author[E. Saor\'in G\'omez]{Eugenia Saor\'in G\'omez}
\address{ALTA Institute for Algebra, Geometry, Topology and their Applications,
Universit\"at Bremen,
D-28359 Bremen, Germany} \email{esaoring@uni-bremen.de}

\author[J. Yepes Nicol\'as]{Jes\'us Yepes Nicol\'as}
\address{Departamento de Matem\'aticas, Universidad de Murcia, Campus de
Espinar\-do, 30100-Murcia, Spain} \email{jesus.yepes@um.es}

\thanks{Both authors are supported by ``Programa de Ayudas a Grupos de Excelencia de
la Regi\'on de Murcia'', Fundaci\'on S\'eneca, 19901/GERM/15.
Second author is supported by MINECO/FEDER project
MTM2015-65430-P and MICINN/FEDER project PGC2018-097046-B-I00.}

\subjclass[2010]{52A20, 52A40}

\keywords{Steiner formula, Minkowski sum, relative quermassintegrals, convex sequences}

\begin{abstract}
We show a reverse isoperimetric inequality within the class of re\-lative outer parallel bodies, with respect to a general convex body $E$, along with its equality condition. Based on the convexity of the sequence of quermassintegrals of Minkowski sums we also prove further inequalities.
\end{abstract}

\maketitle

\section{Introduction}
Let $\K^n$ denote the set of all convex bodies in $\R^n$, i.e., the set of all non-empty
compact convex subsets of $\R^n$. For two
convex bodies $K,E\in\K^n$ and a non-negative real number $\lambda$, the
volume $\vol(\cdot)$ (i.e., the Lebesgue measure) of the Minkowski sum $K+\lambda\,E$ is expressed as a polynomial of
degree at most $n$ in $\lambda$, and it is written as
\begin{equation}\label{eq:steiner-vol}
\vol(K+\lambda E)=\sum_{i=0}^n\binom{n}{i}\W_i(K;E)\lambda^i.
\end{equation}
This expression is called {\em Minkowski-Steiner formula} or {\em relative
Steiner formula} of $K$. The coefficients $\W_i(K;E)$ are the {\em
relative quermassintegrals} of $K$, and they are a special case of the
more general defined {\em mixed volumes} for which we refer to
\cite[Chapter~6]{Gruber} and \cite[Section~5.1]{Schneider}.
In particular, we have $\W_0(K;E)=\vol(K)$, $\W_n(K;E)=\vol(E)$,
$\W_i(\lambda_1 K;\lambda_2 E)=\lambda_1^{n-i}\lambda_2^{i}\W_i(K;E)$ for $\lambda_1,\lambda_2\geq0$
and, if $E$ has dimension $\dim(E)=n$, $\W_i(K;E)=0$ if and only if $\dim(K)\leq n-i-1$.
Moreover, $n\W_1(K;E)$ is the (relative) Minkowski content of $K$, which will be denoted by $\S(K;E)$.
Quermassintegrals admit also a Steiner formula, namely
\begin{equation}\label{eq:steiner-quermass}
\W_k(K+\lambda E;E)=\sum_{i=0}^{n-k}\binom{n-k}{i}\W_{k+i}(K;E)\lambda^i,
\end{equation}
for any $0\leq k\leq n$.

\smallskip

A finite sequence of real numbers $(a_0,\dots,a_m)$ is called concave (see e.g. \cite[Section 7.4]{Schneider}) if
\begin{equation}\label{e: conc seq}
a_{i-1}-2a_i+a_{i+1}\leq 0 \text{  for } i=1,\dots, m-1,
\end{equation}
or equivalently, if
\begin{equation*}\label{e: conc seq dif}
a_0-a_1\leq a_1-a_2\leq \dots\leq a_{m-1}-a_m.
\end{equation*}
Moreover, a sequence $(a_0,\dots,a_m)$ is said to be convex if it satisfies the reversed inequality to \eqref{e: conc seq}, i.e., $a_{i-1}-2a_i+a_{i+1}\geq 0$, for $1\leq i\leq m-1$.

As a consequence of this, a concave sequence satisfies the following inequality \cite[(7.59)]{Schneider}:
\begin{equation}\label{e: conc seq dif2}
(k-j)a_i+(i-k)a_j+(j-i)a_k\leq 0
\end{equation}
for any $0\leq i<j<k\leq m$.
In the latter, there is equality if and only if
\[
a_{r-1}-2a_r+a_{r+1}=0 \text{  for } r=i+1,\dots,k-1.
\]

When the sequence $(a_0,\dots,a_m)$ contains only positive numbers, it is called log-concave if the logarithm of the sequence, namely $(\log a_0,\dots,\log a_m)$ is concave, which is equivalent to the following inequalities \cite[Section 7.4]{Schneider}:
\begin{equation}\label{e: log_concave}
a_i^2\geq a_{i-1}a_{i+1} \text{  for } i=1,\dots, m-1.
\end{equation}

For a convex body $K\in\K^n$ the fundamental inequalities
\begin{equation*}\label{e: AF quermass}
\W_i(K;E)^2\geq \W_{i-1}(K;E)\W_{i+1}(K;E),
\end{equation*}
which are consequences of the more general Aleksandrov-Fenchel inequalities for mixed volumes
(see e.g. \cite[Section 7.3]{Schneider}) do yield that the quermassintegrals of $K$ (relative to $E$)
$\left(\W_0(K;E), \dots, \W_n(K;E)\right)$ constitute a log-concave sequence, i.e., they satisfy
\eqref{e: log_concave}.
Furthermore, when dealing with $n$-dimensional convex bodies $K,E\in\K^n$ with a common projection onto a hyperplane, then the sequence of relative quermassintegrals $(\W_0(K;E),\dots,\W_n(K;E))$ is also concave and thus
it satisfies \eqref{e: conc seq dif2} (see \cite[Theorem~7.7.1 and (7.190)]{Schneider}).

Here we show that, moreover, the reverse form of \eqref{e: conc seq dif2} holds when assuming that $E$ is a
\emph{summand} of $K$:
\begin{teor}\label{t:1}
Let $M,E\in\K^n$ with $K=M+E$ and $\dim(E)=n$.
Then, for every $0\leq i<j<k\leq n$,
\begin{equation}\label{e:1}
(k-j)\W_{i}(K;E)+(i-k)\W_{j}(K;E)+(j-i)\W_{k}(K;E)\geq0.
\end{equation}
Equality holds if and only if $\dim(M)\leq1$.
\end{teor}
By taking $i=0$, $j=1$ and $k=n$ in Theorem \ref{t:1} we get the following result:
\begin{corollary}\label{c:1}
Let $M,E\in\K^n$ with $K=M+E$ and $\dim(E)=n$. Then
\begin{equation*}\label{e: reverse isop}
\vol(K)\geq \frac{S(K;E)}{n-1}-\frac{\vol(E)}{n-1}.
\end{equation*}
Equality holds if and only if $\dim(M)\leq1$.
\end{corollary}

The latter inequality can be regarded as a reverse form of the well-known (relative) \emph{isoperimetric inequality} for $n$-dimensional convex bodies $K,E\in\K^n$ (also referred to in the literature as the \emph{Minkowski first inequality}, see e.g. \cite[Theorem~7.2.1]{Schneider}):
\[S(K;E)^{n}\geq n^n\vol(K)^{n-1}\vol(E),\]
and equality holds if and only if $K=rE$ (up to translations) for some $r>0$.
Another result in this regard is the celebrated reverse isoperimetric inequality, due to Ball \cite{Ba} (with equality
conditions later supplied by Barthe \cite{Bar}), in the classical setting (see also \cite{BaSur} and the references therein).

\smallskip

The case $E=\bigl(1/\lambda\bigr)B_n$ of the Euclidean ball of radius $1/\lambda$ (for $\lambda>0$)
in both Theorem \ref{t:1} and Corollary \ref{c:1}
has been recently obtained in \cite{Tatarko}. There, using Kubota's formula \cite[(5.72)]{Schneider} and induction on the dimension, the authors derive the above relations for the classical quermassintegrals of $K$, i.e., when the relative body $E$ is the Euclidean unit ball $B_n$. Here
we show, by just exploiting the use of the Steiner formula \eqref{eq:steiner-quermass},
that these relations can be extended to the setting of the so-called \emph{Minkowski relative geometry}, this is, when the functionals of convex bodies are evaluated in relation to an arbitrary convex body $E$ rather than $B_n$. This will be shown in Section \ref{s: ineq} of the present paper. Next, in Section \ref{s: further ineq} we prove further inequalities involving quermassintegrals of Minkowski sums of convex bodies.

\section{Relations for the quermassintegrals of Minkowski sums}\label{s: ineq}
We start this section by showing that the sequence $(\W_0(K;E),\dots,\W_n(K;E))$ of (relative) quermassintegrals
is convex when $E$ is a summand of $K$. This will be proven by using the Steiner formula \eqref{eq:steiner-quermass}
jointly with the following property of the binomial coefficients of the numbers $n,k\in\N\cup\{0\}$:
\begin{equation}\label{e: binom}
\binom{n+1}{k}=\binom{n}{k-1}+\binom{n}{k},
\end{equation}
where $\binom{n}{k}:=0$ whenever $k>n$.

\begin{lema}\label{p:der_n-2_i}
Let $M,E\in\K^n$ with $K=M+E$ and $\dim(E)=n$.
Then the sequence $(\W_0(K;E),\dots,\W_n(K;E))$ is convex, i.e., for every $1\leq i\leq n-1$,
\begin{equation}\label{e:der_n-2_i}
\W_{i-1}(K;E)-2\W_i(K;E)+\W_{i+1}(K;E)\geq0.
\end{equation}
Equality holds if and only if $\dim(M)\leq1$.
\end{lema}
\begin{proof}
By the Steiner formula for the quermassintegrals \eqref{eq:steiner-quermass} we get
\begin{equation*}\label{e:usinSt}
\begin{split}
&\W_{i-1}(K;E)-2\W_i(K;E)+\W_{i+1}(K;E)\\
&=\sum_{j=0}^{n-i+1}\binom{n-i+1}{j}\W_{i-1+j}(M;E)
-2\sum_{j=0}^{n-i}\binom{n-i}{j}\W_{i+j}(M;E)\\
&+\sum_{j=0}^{n-i-1}\binom{n-i-1}{j}\W_{i+1+j}(M;E)
=\W_{i-1}(M;E)+\W_{i}(M;E)(n-i-1)\\
&+\sum_{j=1}^{n-i-1}\binom{n-i-1}{j}\W_{i+j}(M;E)
\left(\binom{n-i+1}{j+1}-2\binom{n-i}{j}+\binom{n-i-1}{j-1}\right),
\end{split}
\end{equation*}
and thus, from \eqref{e: binom}, we have that
\begin{equation*}
\begin{split}
\W_{i-1}(K;E)-2\W_i(K;E)+\W_{i+1}(K;E)=\W_{i-1}(M;E)+\W_{i}(M;E)(n-i-1)\\
+\sum_{j=1}^{n-i-2}\binom{n-i-1}{j}\binom{n-i-1}{j+1}\W_{i+j}(M;E)\geq0.
\end{split}
\end{equation*}
Moreover, from the latter identity we may assert that $\W_{i-1}(K;E)-2\W_{i}(K;E)+\W_{i+1}(K;E)=0$ if and only if $\W_{n-2}(M;E)=0$. This completes the proof.
\end{proof}

At this point we would like to compare this result with Bonnesen's inequality in the plane, which establishes that
\begin{equation*}\label{e:Bonnesen}
\W_0(K;E)-2\W_1(K;E)\inr(K;E)+\W_2(K;E)\inr(K;E)^2\leq 0,
\end{equation*}
with equality if and only if $K=L+\inr(K;E)E$ for $\dim L\leq1$.
Here  $\inr(K;E)$ is the {\em relative inradius} of $K$ with respect
to $E$, which is defined by
\[
\inr(K;E)=\sup\{r\geq0:\exists\,x\in\R^n\text{ with } x+r\,E\subset
K\}.
\]
Bonnesen proved this result for $E=B_2$ \cite{Bo}, being the proof of the general case due to Blaschke
\cite[Pages~33-36]{Bla}. This inequality sharpens (in the plane) the Aleksandrov-Fenchel and the isoperimetric inequality, and there is no known generalization of it to higher dimension. Some Bonnesen style inequalities in arbitrary dimension $n$ can be found in e.g. \cite{Bokowski,Brannen,Diskant,Diskant2,HCS,SY} (see also
\cite[Notes for Section~7.2]{Schneider} and the references therein).
\begin{remark}
We notice that if $K,E\in\K^2$ are such that $K=M+\inr(K;E)E$, which implies that $\dim(M)\leq1$,
then Lemma \ref{p:der_n-2_i} for $n=2$ (and such a pair of convex bodies) is just a simple consequence of Bonnesen's inequality.
\end{remark}

\smallskip

Now we derive a more general result for three non-necessarily consecutive (relative) quermassintegrals.
For the sake of clarity, we notice that along its proof we are following the steps of \cite[Pages~399,~400]{Schneider}
(cf. \eqref{e: conc seq dif2}) (see also \cite{Tatarko}).
From now on, and unless we say the opposite, $\W_i:=\W_i(K;E)$.

\begin{proof}[Proof of Theorem \ref{t:1}]
Let $0\leq i<j<k\leq n$ be fixed.
From Lemma \ref{p:der_n-2_i} it follows that
\begin{equation}\label{e:dif>0}
\W_{i}-\W_{i+1}\geq \W_{j-1}-\W_j\geq \W_{k-1}-\W_k\geq 0.
\end{equation}
We notice that the latter inequality follows from the fact that $K=M+E$ and the monotonicity and translation invariance of mixed volumes.

Thus,
\[
\begin{split}
\bigl(\W_{j}-\W_{j+1}\bigr)+\dots+\bigl(\W_{k-1}-\W_k\bigr)&\leq (k-j)\bigl(\W_{j-1}-\W_{j}\bigr)\\
\bigl(\W_{i}-\W_{i+1}\bigr)+\dots+\bigl(\W_{j-1}-\W_{j}\bigr)&\geq (j-i)\bigl(\W_{j-1}-\W_{j}\bigr)
\end{split}
\]
and hence,
\[0\leq\frac{\W_j-\W_{k}}{k-j}\leq\W_{j-1}-\W_{j}\leq\frac{\W_i-\W_j}{j-i},\]
which yields \eqref{e:1}.
Moreover, \eqref{e:1} holds with equality if and only if the same holds for \eqref{e:der_n-2_i}. Thus, the result follows from the equality case of Lemma \ref{p:der_n-2_i}.
\end{proof}

The following classical result provides us with a sufficient condition for a sequence of positive real numbers to be
the (relative) quermassintegrals of certain convex bodies $K,E\in\K^n$:
\begin{prop}[\cite{She}]\label{p:shephard charact}
If the sequence of positive real numbers $(a_0,\dots,a_n)$ is log-concave, then there exist
simplices $K,E\in\K^n$ such that $\W_i(K;E)=a_i$, $0\leq i\leq n$.
\end{prop}

A refined version of this result can be found in \cite{HHCS}. The most general problem of characterizing whether a finite sequence of non-negative numbers are the mixed volumes of $m\in\N$ convex bodies remains open, except the $2$-dimensional case for $m=3$, which was solved by Heine in \cite{Heine}.

\smallskip

Using Proposition \ref{p:shephard charact} we may assure on one hand that there exist convex bodies in $\R^n$
for which \eqref{e:1} does not hold, and that there exist convex bodies, which are not Minkowski sums one of each other, satisfying \eqref{e:1}, on the other hand. We collect the statement here, for the sake of completeness:

\begin{prop}\hfill\label{p:simplices}
\begin{enumerate}
\item There are convex bodies for which \eqref{e:1} does not hold.

\item There exist pairs of convex bodies $K,E\in\K^n$ for which \eqref{e:1} does hold and such that $K\neq M+E$ for all $M\in\K^n$.
\end{enumerate}
\end{prop}

\begin{proof}
The idea for both assertions is to use Proposition \ref{p:shephard charact}.
To prove (i) it is enough to find a log-concave sequence of positive 
real numbers which is not convex. We consider $(1,2,1)$, whose log-concavity (cf. \eqref{e: log_concave}) is clearly fulfilled and thus,
it ensures the existence of planar simplices $K,E\in\K^2$ with $\W_0(K;E)=\W_2(K;E)=1$ and $\W_1(K;E)=2$,
and for which \eqref{e:1} does not hold.

To prove (ii), we argue in the same way with the sequence $(3/4,1/2,1/4)$, which is both log-concave and convex. Since this sequence satisfies \eqref{e:der_n-2_i} with equality, and $\vol(K)=3/4\neq1/4=\vol(E)$, the simplex $E$ cannot be a summand of the simplex $K$.
\end{proof}

Taking Theorem \ref{t:1} into account, together with the log-concavity of the quermassintegrals of any convex body, it is natural to ask whether the log-concavity together with the convexity (cf. Theorem \ref{t:1}) of the sequence of quermassintegrals of a Minkowski sum could provide us with better inequalities or, eventually, equalities in some known inequalities. Unfortunately, so far we have no concluding answers to this issue.

\section{Further inequalities}\label{s: further ineq}
In the following we will write, for $K,E\in\K^n$,
\[f_{K;E}(z)=\sum_{i=0}^{n}\binom{n}{i}\W_{i}(K;E)z^i\]
to denote the (relative) Steiner polynomial of $K$, regarded as a formal polynomial in a complex variable $z\in\C$.
Notice that, for $z\geq0$, $f_{K;E}(z)$ yields the volume of $K+zE$ (cf. \eqref{eq:steiner-vol}).
Moreover, we consider the derivatives of Steiner polynomials in the variable $z$, namely \[f_{K;E}^{(j)}(z)=n(n-1)\cdots(n-j+1)\sum_{k=0}^{n-j}\binom{n-j}{k}\W_{j+k}(K;E)z^k,\] for $0\leq j\leq n$.
Of particular interest for our purposes will be the quantities $f_i^{(j)}$ given by
\begin{equation}\label{e: fij}
f_i^{(j)}=\sum_{k=0}^{n-j}\binom{n-j}{k}\W_{i+k}(K;E)(-1)^k
\end{equation}
for any $0\leq i\leq j\leq n$ (and fixed $K,E\in\K^n$), which can be regarded as formal extensions (up to the constant
$n(n-1)\cdots(n-j+1)$) of the value at $z=-1$ of the $j$-th derivative of the Steiner polynomial, since \[f_j^{(j)}=n(n-1)\cdots(n-j+1)f_{K;E}^{(j)}(-1).\]

\smallskip

Now we prove some generalizations of the inequalities we have dealt with in the previous section.
To avoid straightforward (but a bit lengthy) computations throughout the (proof of) Proposition \ref{p: f_i(j)>0},
first we collect here some direct relations for combinatorial numbers.

\begin{lema}\label{l:ineqs_binom}
Let $N,I,m\in\N\cup\{0\}$. If $I\geq m$ then
\[
\binom{N}{I}-\binom{N-1}{I-1}\binom{m}{1}+\dots+
(-1)^{m-1}\binom{N-m+1}{I-m+1}\binom{m}{m-1}+(-1)^m\binom{N-m}{I-m}\geq0.
\]
Moreover, if $I<m$ then
\[
\begin{split}
\binom{N}{I}-\binom{N-1}{I-1}\binom{m}{1}+\dots+(-1)^I\binom{N-I}{0}\binom{m}{I}\geq0.
\end{split}
\]
\end{lema}
\begin{proof}
Assume first that $I\geq m$. Then, by using recursively \eqref{e: binom}, we have
\[
\begin{split}
&\binom{N}{I}-\binom{N-1}{I-1}\binom{m}{1}+\dots+
(-1)^{m-1}\binom{N-m+1}{I-m+1}\binom{m}{m-1}+(-1)^m\binom{N-m}{I-m}\\
&=\binom{N-1}{I}-\binom{N-2}{I-1}\binom{m-1}{1}+\dots+
(-1)^{m-2}\binom{N-m+1}{I-m+2}\binom{m-1}{m-2}\\
&+(-1)^{m-1}\binom{N-m}{I-m+1}=\dots=\binom{N-m+1}{I}-\binom{N-m}{I-1}=\binom{N-m}{I}\geq0.
\end{split}
\]
Now, if $I<m$, and using again recursively \eqref{e: binom} we get that
\[
\begin{split}
&\binom{N}{I}-\binom{N-1}{I-1}\binom{m}{1}+\dots+(-1)^I\binom{N-I}{0}\binom{m}{I}=\binom{N-1}{I}\\
&-\binom{N-2}{I-1}\binom{m-1}{1}+\dots+(-1)^{I-1}\binom{N-I+1}{1}\binom{m-1}{I-1}\\
&+(-1)^I\binom{N-I-1}{0}\binom{m-1}{I}=\dots=\binom{N-m+I}{I}-\binom{N-m+I-1}{I-1}\binom{I}{1}\\
&+\dots+(-1)^{I-1}\binom{N-m+1}{1}\binom{I}{I-1}+(-1)^I\binom{N-m}{0},
\end{split}
\]
which, by the previous case, is non-negative. This finishes the proof.
\end{proof}

We note that, in terms of the quantities $f_i^{(j)}$ (see \eqref{e: fij}), Lemma \ref{p:der_n-2_i} yields that \[f_i^{(n-2)}\geq0\] for every $0\leq i\leq n-2$. Here we extend these inequalities to any value of $j$, with
$0\leq i\leq j\leq n$.
\begin{prop}\label{p: f_i(j)>0}
Let $M,E\in\K^n$ with $K=M+E$ and $\dim(E)=n$.
Then, for every $0\leq i\leq j\leq n$,
\begin{equation*}\label{e:der_j_i}
f_i^{(j)}\geq 0.
\end{equation*}
Equality holds if and only if $\dim(M)\leq n-j-1$.
\end{prop}
\begin{proof}
From the Steiner formula for quermassintegrals \eqref{eq:steiner-quermass}, and writing $m=n-j$, we have
\[
\begin{split}
f_i^{(j)}&=\sum_{k=0}^m\binom{m}{k}\left(\sum_{r=0}^{n-i-k}\binom{n-i-k}{r}\W_{i+k+r}(M;E)\right)(-1)^k\\
&=\sum_{r=0}^{n-i}\binom{n-i}{r}\W_{i+r}(M;E) -
m\sum_{r=0}^{n-i-1}\binom{n-i-1}{r}\W_{i+1+r}(M;E)+
\dots\\
&+(-1)^m\sum_{r=0}^{n-i-m}\binom{n-i-m}{r}\W_{i+m+r}(M;E)\\
&=\sum_{r=0}^{n-i-m}\W_{i+m+r}(M;E)\left(\binom{n-i}{r+m}\binom{m}{0}-\binom{n-i-1}{r+m-1}\binom{m}{1}+\dots\right.\\
&\left.+(-1)^{m-1}\binom{n-i-m+1}{r+1}\binom{m}{m-1}+(-1)^m\binom{n-i-m}{r}\binom{m}{m}\right)\\
&+\W_{i+m-1}(M;E)\left(\binom{n-i}{m-1}\binom{m}{0}-\binom{n-i-1}{m-2}\binom{m}{1}+\dots\right.\\
&\left.+(-1)^{m-2}\binom{n-i-m+2}{1}\binom{m}{m-2}+(-1)^{m-1}\binom{n-i-m+1}{0}\binom{m}{m-1}\right)\\
&+\dots+\W_{i+1}(M;E)\left(\binom{n-i}{1}\binom{m}{0}-\binom{n-i-1}{0}\binom{m}{1}\right)+\W_i(M;E).
\end{split}
\]
From Lemma \ref{l:ineqs_binom} we obtain that all the summands in the above expression are
non-negative. Moreover, since the sole non-zero coefficients involved are those with indexes between
$i$ and $j$, we have that $f_i^{(j)}=0$ if and only if $\W_{j}(M;E)=0$. This completes the proof.
\end{proof}

\begin{lema}\label{l: diff i der}
Let $M,E\in\K^n$ with $K=M+E$ and $\dim(E)=n$.
Then, for every $0\leq i<j\leq n$,
\begin{equation*}\label{e:id der}
f_i^{(j)}-f_{i+1}^{(j)}=f_i^{(j-1)}\geq 0.
\end{equation*}
\end{lema}

\begin{proof}
By using \eqref{e: binom} we get that
\[
\begin{split}
f_i^{(j)}-f_{i+1}^{(j)}&=\sum_{k=0}^{n-j}\binom{n-j}{k}(-1)^k\W_{i+k}-\sum_{k=0}^{n-j}\binom{n-j}{k}(-1)^k\W_{i+1+k}\\
&=\W_i+\sum_{l=1}^{n-j}(-1)^l\W_{i+l}\left[\binom{n-j}{l}+\binom{n-j}{l-1}\right]
+(-1)^{n-j+1}\W_{i+1+n-j}\\
&=\W_i+\sum_{l=1}^{n-j}(-1)^l\W_{i+l}\binom{n-j+1}{l}
+(-1)^{n-j+1}\W_{i+1+n-j}=f_i^{(j-1)},
\end{split}
\]
which, jointly with Proposition \ref{p: f_i(j)>0}, finishes the proof.
\end{proof}

We conclude the paper by showing an extension of Theorem \ref{t:1}. Indeed, the latter is derived by setting $l=n$ in
the following result.
\begin{teor}
Let $M,E\in\K^n$ with $K=M+E$ and $\dim(E)=n$.
Then, for every $0\leq i<j<k\leq l\leq n$,
\begin{equation}\label{e:2}
(k-j)f_i^{(l)}+(i-k)f_{j}^{(l)}+(j-i)f_k^{(l)}\geq0.
\end{equation}
Equality holds if and only if $\dim(M)\leq n-l+1$.
\end{teor}

\begin{proof}
The proof is completely analogous to the proof of Theorem \ref{t:1}, interchanging every involved $r$-th quermassintegrals by $f_r^{(l)}$, and  using Lemma \ref{l: diff i der} instead of Lemma \ref{p:der_n-2_i}.
Finally, \eqref{e:2} holds with equality if and only if the same holds for
\[f_{j-1}^{(l)}-f_j^{(l)}=f_{j-1}^{(l-1)}\geq f_{j}^{(l-1)}=f_{j}^{(l)}-f_{j+1}^{(l)},\]
which is equivalent to $f_{j-1}^{(l-2)}\geq0$. Then the result follows from the equality case of
Proposition \ref{p: f_i(j)>0}.
\end{proof}





\noindent {\it Acknowledgements.} We would like to thank the anonymous referee for his/her very valuable
suggestions which have allowed us to considerably improve the presentation of the manuscript.


\begin{thebibliography}{13}

\bibitem{Ba} K. Ball, Volume ratios and a reverse isoperimetric inequality,
{\it J. London Math. Soc.} {\bf 44} (1991), 351-359.

\bibitem{BaSur} K. Ball, An elementary introduction to modern convex geometry. In
{\it Flavors of Geometry}, Math. Sci. Res. Inst. Publ. 31, pp. 1-58,
Cambridge Univ. Press, Cambridge, 1997.

\bibitem{Bar} F. Barthe, On a reverse form of the Brascamp-Lieb inequality,
{\it Invent. Math.} {\bf 134} (1998), 335-361.


\bibitem{Bla} W. Blaschke, {\it Vorlesungen \"uber Integralgeometrie}. Third edition. Deutscher
Verlag der Wissenschaften, Berlin, 1995 (First edition: 1949).

\bibitem{Bokowski} J. Bokowski, Eine versch\"arfte Ungleichung zwischen Volumen, Oberfl\"ache und
Inkugelradius im $R^n$, {\it Elem. Math.} {\bf 28} (1973), 43-44.

\bibitem{Bo} T. Bonnesen, {\it Les probl\`emes des isop\'erim\`etres et des is\'epiphanes}.
Collection de monographies sur la th\'eorie des fonctions, Gauthier-Villars, Paris 1929.


\bibitem{Brannen} N. S. Brannen, The Wills conjecture, {\it Trans. Amer. Math. Soc.} {\bf 349} (1997), 3977-3987.

\bibitem{Tatarko} R. Chernov, K. Drach and K. Tatarko, A sausage body is a unique solution for
a reverse isoperimetric problem, {\it Adv. Math.} {\bf 353} (2019), 431-445.

\bibitem{Diskant} V. I. Diskant, A generalization of Bonnesen's inequalities (in Russian),
{\it Dokl. Akad. Nauk SSSR} {\bf 213} (1973), 519-521. English translation:
{\it Soviet Math. Dokl.} {\bf 14} (1973), 1728-1731.

\bibitem{Diskant2} V. I. Diskant, Strengthening of an isoperimetric inequality (in Russian),
{\it Sibirskii Mat. Zh.} {\bf 14} (1973), 873-877. English translation:
{\it Siberian Math. J.} {\bf 14} (1973), 608-611.

\bibitem{Gruber} P. M. Gruber, {\it Convex and Discrete Geometry}.
Springer, Berlin Heidelberg, 2007.


\bibitem{Heine} R. Heine, Der Wertevorrat der gemischten Inhalte von zwei, drei und vier ebenen
Eibereichen, {\it Math. Ann.} {\bf 115} (1937), 115-129.

\bibitem{HHCS} M. Henk, M. A. Hern\'andez Cifre and E. Saor\'in, Steiner polynomials via ultra-logconcave sequences,
{\it Commun. Contemp. Math.} {\bf 14} (6) (2012), 1-16.

\bibitem{HCS} M. A. Hern\'andez Cifre and E. Saor\'in, On inner parallel bodies and quermassintegrals,
{\it Israel J. Math.} {\bf 177} (2010), 29-47.

\bibitem{SY} J. R. Sangwine-Yager, Bonnesen-style inequalities for Minkowski relative geometry,
{\it Trans. Amer. Math. Soc.} {\bf 307} (1988), 373-382.

\bibitem{Schneider} R. Schneider, {\it Convex Bodies: The Brunn-Minkowski
Theory}. Second expanded edition. Cambridge University Press, Cambridge,
2014.

\bibitem{She} G. C. Shephard, Inequalities between mixed volumes of convex
sets, {\it Mathematika} {\bf 7} (1960), 125--138.

\end{thebibliography}
\end{document}